\documentclass[12pt,preprint]{amsart}

\usepackage{amssymb,amsfonts,amsthm,amsmath,amscd}

\theoremstyle{plain}
\newtheorem{proposition}{Proposition}%[section]
\newtheorem{lemma}[proposition]{Lemma}
\newtheorem{theorem}[proposition]{Theorem}
\newtheorem{corollary}[proposition]{Corollary}

\newtheorem{conjecture}[proposition]{Conjecture}
\newtheorem{question}[proposition]{Question}
\newtheorem{problem}[proposition]{Problem}

\theoremstyle{definition}
\newtheorem*{ack}{Acknowledgement}

\theoremstyle{remark}

\def\Z{\mathbb{Z}}
\def\NN{\mathbb{N}}

\def\P{\mathcal{P}}
\def\N{\mathcal{N}}
\def\G{\mathcal{G}}
\def\F{\mathcal{F}}
\def\R{\mathcal{R}}
\def\E{\mathcal{E}}

\def\A{\mathcal{A}}
\def\B{\mathcal{B}}

\begin{document}

\title{Variants of Wythoff's game translating its $\P$-positions}

\author{Nhan Bao Ho}
\address{Department of Mathematics, La Trobe University, Melbourne, Australia 3086}
\email{nbho@students.latrobe.edu.au, honhanbao@yahoo.com}

\begin{abstract}
We introduce a restriction of Wythoff's game, which we call $\F$-\emph{Wythoff}, in which the integer ratio of entries must not change if an equal number of tokens are removed from both piles. We show that $\P$-positions of $\F$-Wythoff are exactly those positions obtained from $\P$-positions of Wythoff's game by adding 1 to each entry. We describe the distribution of Sprague-Grundy values and, in particular, generalize two properties on the distribution of those positions which have Sprague-Grundy value $k$, for a given $k$, for variants of Wythoff's game. We analyze the  mis\`{e}re $\F$-Wythoff and show that the normal and  mis\`{e}re versions differ exactly on those positions which have Sprague-Grundy values 0, and 1 via a swap. We examine two further variants of $\F$-Wythoff, one restriction and one extension, preserving its $\P$-positions. We raise two general questions based on the translation phenomenon of the $\P$-positions.
\end{abstract}

%\begin{keyword}
%Wythoff's game, $\P$-positions, Sprague-Grundy values, miserability, combinatorial games,
%\end{keyword}

\maketitle

\section{Introduction}

Introduced by Willem Abraham Wythoff \cite{Wyt}, Wythoff's game is a variant of Nim played on two piles of tokens. Two players move alternately, either removing a number of tokens from one pile or removing an equal number of tokens from both piles. The player first unable to move (because the two piles become empty) loses and his/her opponent wins. We denote by $(a,b)$ the position of the two piles of sizes $a$, $b$. Symmetrically, $(a,b)$ is identical to $(b,a)$. A position is called an \emph{$\N$-position} (known as \emph{winning position}) if the player about to move from there has a strategy to win. Otherwise, we have a \emph{$\P$-position} (known as \emph{losing position}). Here, $\N$ stands for the $\N$ext player and $\P$ stands for the $\P$revious player. Wythoff \cite{Wyt} shows that the $\P$-positions of Wythoff's game form the set $\{(\lfloor \phi n \rfloor, \lfloor \phi^2 n \rfloor) | n \geq 0\}$ where $\phi = (1+\sqrt{5})/2$ is the golden ratio and $\lfloor . \rfloor$ denotes the integer part.

Recall that Wythoff's game is an impartial combinatorial game. If there exists a move from a position $p$ to some position $q$, then the position $q$ is called a \emph{follower} of $p$. For a finite set $S$ of nonnegative integers, the \emph{minimum excluded number} of $S$, denoted by $mex(S)$, is the smallest nonnegative integer not in $S$. The \emph{Sprague-Grundy function} for an impartial combinatorial game is the function $\mathcal{G}$ from the set of its positions into the nonnegative integers, defined inductively by
\[\mathcal{G}(p) = mex\{\mathcal{G}(q)| q \text{ is a follower of } p\}\]
with $mex\{\} = 0$. The value $\mathcal{G}(p)$ is called the \emph{Sprague-Grundy value} at $p$. The base theory of combinatorial games and the Sprague-Grundy function can be found in \cite{ww1}. Recall that a position is a $\P$-position if and only if it has Sprague-Grundy value 0 \cite{ww1}. One also can prove the following fundamental property of Sprague-Grundy values: a position $p$ has Sprague-Grundy value $k > 0$ if and only if (i) $\G(p) \neq \G(q)$ if there exists a move from either $p$ to $q$ or $q$ to $p$, and (ii) for every $l < k$, there exists a position $q$ such that $\G(q) = l$ and one can move from $p$ to $q$.

Despite being one of oldest impartial games, Wythoff's game is still a highly interesting topic. The Sprague-Grundy function for this game is studied widely in \cite{blass, Dress, landman, nivasch}. Several variants of Wythoff's game have been examined, including (i) {\em restrictions}: obtained from Wythoff's game by eliminating some moves \cite{Gen-Connell, Nim-Wythoff, ho}, and (ii) {\em extensions}: obtained from Wythoff's game by  adding extra moves \cite{Heapgame, Howtobeat, Gen-Fra, Adjoining, RRR0, RRR, ho, hog, Some-Gen}. Solving the winning strategy for variants of Wythoff's game is always an interesting exercise. In particular, it has been showed that there exist variants of Wythoff's game, including restrictions and extensions, preserving its $\P$-positions \cite{Ext-Res, ho}. This paper makes further investigations of variants of Wythoff's game whose $\P$-positions are slightly different to those of Wythoff's game.

In this paper, we introduce a restriction of Wythoff's game, called \emph{$\F$-Wythoff}, in which a legal move is either of the following two types:
\begin{itemize}
\item [(i)] removing any number of tokens from one pile;
\item [(ii)] removing an equal number of tokens from two piles provided that the integer ratio of the two entries does not change.
\end{itemize}
We obtain the following result on $\P$-positions: the $\P$-positions of $\F$-Wythoff are those positions obtained directly from $\P$-positions of Wythoff's game by adding 1 to each entry. This translation phenomenon is the main theme of this paper. We also establish several results for further modifications of $\F$-Wythoff.

The paper is organized as follows. In the next section, we solve the $\P$-positions in both algebraic and recursive characterizations before giving formulas for those positions which have Sprague-Grundy values 1 and 2. Section 3 analyzes the distributions of Sprague-Grundy values for $\F$-Wythoff on the 2-dimension array whose $(i,j)$ entry is the Sprague-Grungdy value of the position $(i,j)$. In particular, we generalize two results for the Sprague-Grundy values of Wythoff's game and its variants. In Section 4, we examine $\F$-Wythoff in mis\`{e}re play. We show that $\F$-Wythoff and its mis\`{e}re form differ slightly on the set of positions which have Sprague-Grundy values 0 and 1. Section 5 answers the question as to whether there exists a variant of $\F$-Wythoff preserving its $\P$-positions. Two such variants, one restriction and an extension, are discussed. In the final section, we raise two questions for variants of Wythoff's game, based on the theme of the paper. This paper is the continuation of our investigations on variants of Wythoff's game \cite{ho} and, more generally, 2-pile variants of the game of Nim \cite{MEuclid, Min, CHL}.

\section{Translation phenomenon on those positions which have Sprague-Grundy value 0, 1 and 2}

This section first solves the $\P$-positions in $\F$-Wythoff. We then give formulas for those positions which have Sprague-Grundy values 1 and 2. It will be shown that these positions are all obtained from $\P$-positions of Wythoff's game by a translation, except for some initial positions.

%====================================================================================================================
%====================================================================================================================
%====================================================================================================================
% P-POSITIONS

Let $\phi = (1+\sqrt{5})/2$. Then $\phi^2 = \phi + 1$. Therefore, for every positive integer $n$, we have
\[\lfloor \phi^2n \rfloor = \lfloor \phi n + n \rfloor = \lfloor \phi n \rfloor + n.\]

%====================================================================================================================
\smallskip
\begin{lemma} \label{Comp} \cite{beatty1}
For each $i \geq 1$, set $a_i = \lfloor \phi i \rfloor$ and $b_i = \lfloor \phi^2 i \rfloor$. Then
\begin{align*}
& \{a_i | i \geq 1\} \cup \{b_i | i \geq 1\} = \NN,\\
& \{a_i | i \geq 1\} \cap \{b_i | i \geq 1\} = \emptyset,
\end{align*}
in which $\NN$ is the set of positive integers.
\end{lemma}

Consequently, we have

%====================================================================================================================
\smallskip
\begin{corollary} \label{Comp.1}
Let $a \geq 2$ be an integer. There exists exactly one $n$ such that either $a = \lfloor \phi n \rfloor+1$ or $a = \lfloor \phi^2 n \rfloor +1$. Moreover, the number $a$ cannot be of both forms.
\end{corollary}

Recall that in Wythoff's game, we have $\P$-positions as follows.
%====================================================================================================================
\smallskip
\begin{theorem} \label{W-P} \cite{Wyt}
A position in Wythoff's game is a $\P$-position if and only if it is of the form $(\lfloor\phi n\rfloor, \lfloor \phi^2 n \rfloor)$ for some $n \geq 0$.
\end{theorem}

We first prove an equality that will be used many times in this paper.
%====================================================================================================================
\smallskip
\begin{lemma} \label{U-lem}
Let $n,k$, and $i$ be positive integers. We have
\[
\Bigg{\lfloor} \frac{\lfloor\phi^2 n\rfloor + k + i}{\lfloor\phi n\rfloor + k + i}\Bigg{\rfloor} = \Bigg{\lfloor} \frac{\lfloor\phi^2 n\rfloor + i}{\lfloor\phi n\rfloor + i} \Bigg{\rfloor} = 1.
\]
\end{lemma}

\begin{proof}
We have
\begin{align*}
  \Bigg{\lfloor} \frac{\lfloor\phi^2 n\rfloor + k + i}{\lfloor\phi n\rfloor + k + i}\Bigg{\rfloor}
& = \Bigg{\lfloor} \frac{\lfloor\phi n\rfloor +n + k + i}{\lfloor\phi n\rfloor + k + i}\Bigg{\rfloor}
= 1 + \Bigg{\lfloor} \frac{n}{\lfloor\phi n\rfloor + k + i}\Bigg{\rfloor}\\
&= 1
= 1 + \Bigg{\lfloor} \frac{n}{\lfloor\phi n\rfloor  + i}\Bigg{\rfloor}
= \Bigg{\lfloor} \frac{\lfloor\phi^2 n\rfloor + i}{\lfloor\phi n\rfloor + i} \Bigg{\rfloor}.
\end{align*}
\end{proof}

We now solve the winning strategy for $\F$-Wythoff.

%====================================================================================================================
\smallskip
\begin{theorem}[Algebraic characterization] \label{FW-P}
A position in $\F$-Wythoff is a $\P$-position if and only if it is an element of the set
\[\{(0,0), (\lfloor \phi n \rfloor+1, \lfloor \phi^2n \rfloor +1) | n \geq 0 \}. \]
\end{theorem}

\begin{proof}

Let $\A = \{(0,0), (\lfloor\phi n\rfloor +1, \lfloor \phi^2 n \rfloor + 1) | n \geq 0\}$. We need to show that the following two properties hold for $\F$-Wythoff:
\begin{itemize}
\item [(i)]  Every move from a position in $\A$ cannot terminate in $\A$,
\item [(ii)] From every position not in $\A$, there is a move terminating in $\A$.
\end{itemize}

For (i), assume by contradiction that there is a move from $(\lfloor\phi n\rfloor +1, \lfloor \phi^2 n \rfloor + 1)$ to $(\lfloor\phi m\rfloor +1, \lfloor \phi^2 m \rfloor + 1)$ for some $n > m$. At the moment, there are at most three possibilities for this move: (1) removing some $k$ tokens from the small pile $\lfloor\phi n\rfloor +1$, or (2) removing some $k$ tokens from the large pile $\lfloor\phi^2 n\rfloor +1$, or (3) removing some $k$ tokens from both piles. Note that the move (3) does not exist. In fact, otherwise one can move from $(\lfloor\phi n\rfloor, \lfloor \phi^2 n \rfloor)$ to $(\lfloor\phi m\rfloor, \lfloor \phi^2 m \rfloor)$ in Wythoff's game. This is impossible. Possibility (1) implies the system
\begin{align*}
\begin{cases}
\lfloor\phi n\rfloor +1 - k = \lfloor\phi m\rfloor +1, \\
\lfloor \phi^2 n \rfloor + 1 = \lfloor \phi^2 m \rfloor + 1.
\end{cases}
\end{align*}
It follows from the second equation that $n = m$, giving a contradiction. Possibility (2) implies the system
\begin{align*}
\begin{cases}
\lfloor \phi^2 n \rfloor + 1 - k = \lfloor\phi m\rfloor +1, \\
\lfloor\phi n\rfloor +1 = \lfloor \phi^2 m \rfloor + 1
\end{cases}
\end{align*}
giving $\lfloor\phi n\rfloor = \lfloor \phi^2 m \rfloor$. This is impossible by Lemma \ref{Comp}.

For (ii), let $p = (a,b) \notin \A$. Set $q = (a-1,b-1)$. Then $q$ is not of the form $(\lfloor\phi n\rfloor, \lfloor\phi^2 n\rfloor)$. By Theorem \ref{W-P}, there exists a legal move from $q$ to some $(\lfloor\phi n\rfloor, \lfloor\phi^2 n\rfloor)$ in Wythoff's game. This move is identical to the move from $p$ to $(\lfloor\phi n\rfloor +1, \lfloor\phi^2 n\rfloor+1)$ in $\F$-Wythoff provided that the integer ratio of the two entries does not change if an equal number of tokens is removed from both piles. The proof is complete by Lemma \ref{U-lem}.
\end{proof}

Set $(a_0,b_0) = (0,0)$. For $n \geq 1$, set $a_n = \lfloor\phi (n-1)\rfloor +1$, $b_n = \lfloor\phi^2 (n-1)\rfloor +1$. Then $\{(a_i,b_i) | i \geq 0\}$ is the set of $\P$-positions of $\F$-Wythoff. We now describe a recursive characterization of the sequence $\{(a_i,b_i)\}_{i \geq 0}$.

%===============================
\smallskip
\begin{corollary}[Recursive characterization] \label{FW-P-A}
Consider the sequence $\{(a_i,b_i)\}_{i \geq 1}$ of $\P$-positions of $\F$-Wythoff. For each $n \geq 1$, we have
\begin{align*}
    \begin{cases}
    a_n = mex\{a_i,b_i | 0 \leq i \leq n-1\},\\
    b_n = a_n+n-1.
    \end{cases}
\end{align*}
\end{corollary}

\begin{proof}
The first equation follows from Corollary \ref{a_i} below and the second equation follows from Theorem \ref{FW-P}. Note that the order that Corollary \ref{a_i} comes in the paper does not affect its independent content used for this corollary.
\end{proof}

%====================================================================================================================
%====================================================================================================================
%====================================================================================================================
% POSITIONS OF SPRAGUE-GRUNDY 1

We next give a formula for those positions which have Sprague-Grundy value 1.

%====================================================================================================================
\smallskip
\begin{theorem} \label{V1}
The set of positions which have Sprague-Grundy value $1$ in $\F$-Wythoff is
\[\B = \{(0,1), (\lfloor \phi n \rfloor +2, \lfloor \phi^2 n \rfloor +2) | n \geq 0\}.\]
\end{theorem}

\begin{proof}
Recall that the set of $\P$-positions of $\F$-Wythoff is
\[\P = \{(0,0),(\lfloor \phi n \rfloor +1, \lfloor \phi^2 n \rfloor + 1)| n \geq 0\}.\]
Based on the definition of Sprague-Grundy function, we need to prove that
\begin{itemize}
\item [(i)] $\B \cap \P = \emptyset$,
\item [(ii)] There is no move from a position in $\B$ to a position in $\B$,
\item [(iii)] From every position not in $\B \cup\P$, there exists a move to some position in $\B$.
\end{itemize}

For (i), assume by contradiction that $\B \cap \P \neq \emptyset$. Then there exist $n,m$ such that
\[(\lfloor \phi n \rfloor +1, \lfloor \phi^2 n \rfloor + 1) = (\lfloor \phi m \rfloor +2, \lfloor \phi^2 m \rfloor +2).\]
It follows that
\begin{align*}
\begin{cases}
\lfloor \phi n \rfloor = \lfloor \phi m \rfloor +1, \\
\lfloor \phi n \rfloor + n = \lfloor \phi m \rfloor + m +1.
\end{cases}
\end{align*}
One can check that this system of equations gives a contradiction.

For (ii), one can check that there is no move from a position of the form $(\lfloor \phi n \rfloor +2, \lfloor \phi^2 n \rfloor +2)$ to (0,1). Similar to case (ii) in Theorem \ref{FW-P}, one can check that there is no move between positions of the form $(\lfloor \phi n \rfloor +2, \lfloor \phi^2 n \rfloor + 2)$.

For (iii), let $(a,b) \notin \B \cup \P$ with $a \leq b$. One can move from $(a,b)$ to (0,1) if either $a = 0$ or $a = 1$. We now assume that $a \geq 2$. Consider the position $p = (a-2,b-2)$. Note that $p$ is not of the form $(\lfloor \phi n \rfloor, \lfloor \phi^2 n \rfloor)$. In Wythoff's game, there is one move from $p$ to some position $(\lfloor \phi m \rfloor, \lfloor \phi^2 m \rfloor)$. This move is identical to the move from $(a,b)$ to $(\lfloor \phi m \rfloor + 2, \lfloor \phi^2 m \rfloor +2)$ in $\F$-Wythoff provided that the integer ratio of the two entries does not change if an equal number of tokens is removed from both piles. The proof is then complete by Lemma \ref{U-lem}.
\end{proof}

Theorem \ref{V1} shows that, except for the first position, those positions which have Sprague-Grundy value 1 are all obtained from $\P$-positions by adding 1 to each entry. Similarly, those positions which have Sprague-Grundy value 2 in $\F$-Wythoff are also obtained from the $\P$-positions via a translation, except for the first two positions. We leave the proof of the following theorem for the readers.

%====================================================================================================================
\smallskip
\begin{theorem} \label{V2}
A position $(a,b)$ has Sprague-Grundy value $2$ in $\F$-Wythoff if and only if $(a,b)$ is an element of the set
\[\{(0,2), (1,3), (\lfloor \phi n \rfloor + 4, \lfloor \phi^2 n \rfloor + 4) | n \geq 0\}.\]
\end{theorem}

We do not know a formula of such forms for those positions which have Sprague-Grundy values more than 2. It would be therefore interesting if we can answer the following question.

%====================================================================================================================
\smallskip
\begin{question}
Does there exist $g > 2$ such that those positions which have Sprague-Grundy value $g$, possibly except for a finite number of positions, can be obtained by a translation from the $\P$-positions?
\end{question}

%====================================================================================================================
%====================================================================================================================
%====================================================================================================================
% MORE ON SPRAGUE-GRUNDY FUNCTION

\section{On the distribution of Sprague-Grundy values}
Consider the 2-dimension infinite array $\mathbb{A}$ whose $(i,j)$ entry is the Sprague-Grundy value $\G(i,j)$ of the position $(i,j)$ in $\F$-Wythoff. Table \ref{T} displays some values of the array with $i, j \leq 9$.
\begin{table} [ht]
\begin{center}
\begin{tabular}{c|cccccccccc}
  9 &9 &8 &11 &10 &12 &13&1 &2 &6 &7 \\
  8 &8 &9 &10 &7  &11 &0 &12&4 &5 &6 \\
  7 &7 &6 &5  &8  &9  &1 &10&11&4 &2 \\
  6 &6 &7 &4  &5  &0  &2 &3 &10&12&1 \\
  5 &5 &4 &7  &6  &3  &8 &2 &1 &0 &13 \\
  4 &4 &5 &6  &1  &2  &3 &0 &9 &11&12 \\
  3 &3 &2 &0  &4  &1  &6 &5 &8 &7 &10 \\
  2 &2 &3 &1  &0  &6  &7 &4 &5 &10&11 \\
  1 &1 &0 &3  &2  &5  &4 &7 &6 &9 &8 \\
  0 &0 &1 &2  &3  &4  &5 &6 &7 &8 &9\\
\hline
i/j &0 &1 &2 &3 &4 &5 &6 &7 &8 &9
 \end{tabular}
\caption{Sprague-Grundy values $\G(i,j)$ for $i,j\leq 9$}\label{T}
\end{center}
\end{table}
We discuss in this section the distribution of Sprague-Grundy values in the array $\mathbb{A}$. We first show that each row (column) in $\mathbb{A}$ contains every Sprague-Grundy value exactly one time.

%====================================================================================================================
\smallskip
\begin{theorem} \label{Row}
Let $a, g$ be nonnegative integers. There exists a unique integer $b$ such that $\G(a,b) = g$.
\end{theorem}

\begin{proof}
The uniqueness holds by the definition of $\F$-Wythoff. We now prove the existence. Note that the theorem holds for $a = 0$. We first show that the theorem holds for $g = 0$. If $a = 1$ then $\G(1,1) = 0$. If $a \geq 2$, by Corollary \ref{Comp.1}, there exists $m$ such that either $a = \lfloor \phi m \rfloor + 1$ or $a = \lfloor \phi^2 m \rfloor + 1$. The former case gives $\G(a,\lfloor \phi^2 m \rfloor + 1) = 0$ and the latter case gives $\G(a,\lfloor \phi m \rfloor + 1) = 0$.

Assume that $g> 0$ and assume by contradiction that the sequence $R_a = \{\G(a,n)\}_{n \geq 0}$ does not contain $g$. We can assume that $g$ is the smallest integer not in the sequence $R_a$. Then there exists the smallest integer $b_0 \geq a$ such that
\begin{align} \label{BW-b0}
\{0, 1, \ldots,g -1\} \subseteq \{\G(a,i) | i \leq b_0-1\}.
\end{align}
For each $s \geq 1$, let $b_s = b_0+s(a+1)$. We have
\begin{align*}
\G(a,b_s) = mex\{&\G(a-i,b_s), \G(a,b_s-j), \G(a-l,b_s-l) | 1 \leq i \leq a,\\
                 &1 \leq j  \leq b_s, 1 \leq l < a, \lfloor\frac{b_s-l}{a-l}\rfloor = \lfloor\frac{b}{a}\rfloor\}.
\end{align*}
By (\ref{BW-b0}), the $mex$ set contains $\{0, 1, \ldots,g -1\}$. Note that $\G(a,b_s) \neq g$. Therefore, $\G(a,b_s) > g$ and so the $mex$ set contains $g$. Since $\G(a,b_s-i) \neq g$ for all $i$, there exists either $i_s \leq a$ or $l_s < a$ such that either $\G(a-i_s,b_s) = g$ or $\G(a-l_s,b_s-l_s) = g$. Note that as $s$ varies, the integers $b_s$ assume infinitely many values, while $i_s, l_s\leq a$ for each $s$. Moreover, by the uniqueness, there are at most $a+1$ positions of the form $(a-i_s,b_s)$ whose Sprague-Grundy values all are $g$. So there must exist $s_1 < s_2$ such that $l_{s_1} = l_{s_2}$ and $\G(a-l_{s_1},b_{s_1}-l_{s_1}) = \G(a-l_{s_2},b_{s_2}-l_{s_2})$. This is impossible since one can move from $(a-l_{s_2},b_{s_2}-l_{s_2})$ to $(a-l_{s_1},b_{s_1}-l_{s_1})$ by removing $b_{s_2}-b_{s_1}$ tokens from the larger pile. Thus, the sequence $R_a$ contains $g$ and so $\G(a,b) = g$ for some $b$.
\end{proof}

Recall that the \emph{Wythoff sequence}, or sequence of $\P$-positions, of Wythoff's game is the sequence $\{(\lfloor \phi n \rfloor, \lfloor \phi^2 n \rfloor)\}_{n \geq 0}$. Let $A = \{a_n | n \geq 0\}$, $B = \{b_n | n \geq 0\}$ where $a_n = \lfloor \phi n \rfloor$, $b_n = \lfloor \phi^2 n \rfloor$. By Lemma \ref{Comp}, the Wythoff sequence of Wythoff's game satisfies conditions $A \cap B = \{0\}$, $A \cup B = \Z_{\geq 0}$ where $\Z_{\geq 0}$ denotes the set of nonnegative integers. Moreover, $a_n = mex\{a_i, b_i | 0 \leq i < n\}$. Curiously, these three conditions hold for several variants of Wythoff's game \cite{Heapgame, Howtobeat, Wyt-mis, Adjoining, ho}. We generalize these results for $\F$-Wythoff in the next two corollaries. The \emph{$k$-sequence} of $\F$-Wythoff is the sequence $\{(a_n,b_n)\}_{n \geq 0}$ of positions whose Sprague-Grundy values are $k$ in which $0 \leq a_n \leq b_n$ and $a_n < a_m$ for $n < m$. We first describe a recursive characterization of the first entries $a_i$ in the $k$-sequence of $\F$-Wythoff.

%====================================================================================================================
\smallskip
\begin{corollary} \label{a_i}
For each $k \geq 0$, consider the $k$-sequence $\{(a_n,b_n)\}_{n \geq 0}$ of $\F$-Wythoff. We have $a_n = mex\{a_i, b_i | 0 \leq i < n\}$ for each $n$.
\end{corollary}
\begin{proof}
Assume that there exists an integer $n > 0$ such that $a_n \neq m = mex\{a_i,b_i | 0 \leq i \leq n-1\}$. If $a_n < m$ then
$a_n \in \{a_i,b_i | 0 \leq i \leq n-1\}$ and so there exists $l < n$ such that $a_n = b_l$. This means there exists a move between the two positions $(a_n,b_n)$ and $(a_l,b_l)$ whose Sprague-Grundy values are $k$. This is a contradiction. Assume now that $a_n > m$. By Theorem \ref{Row}, there exists $m'$ (either $m \leq m'$ or $m' < m$) such that the position $(m,m')$ has Sprague-Grundy value $k$. This means there exists some $j < n$ such that $(m,m')$ is identical to $(a_j, b_j)$. It follows that $m \in mex\{a_i,b_i | 0 \leq i \leq n-1\}$ giving a contradiction. Hence, $a_n = mex\{a_i,b_i | 0 \leq i \leq n-1\}$ for all $n$.
\end{proof}

Note that Corollary \ref{a_i} is still true for those variants of Wythoff's game which satisfy Theorem \ref{Row}. Such games, including Wythoff's game, are discussed in \cite{ho}. Similarly, the first equation in the next corollary also holds for those variants of Wythoff's game which satisfy Theorem \ref{Row}.

%====================================================================================================================
\smallskip
\begin{corollary} \label{Z+}
For each $k \geq 0$, consider the $k$-sequence $\{(a_n,b_n)\}_{n \geq 0}$ of $\F$-Wythoff. We have
\begin{align*}
\begin{cases}
\{a_n | n \geq 0\} \cup \{b_n | n \geq 0\} = \Z_{\geq 0}, \\
|\{a_n | n \geq 0\} \cap \{b_n | n \geq 0\}| \leq 2,
\end{cases}
\end{align*}
where $|S|$ is the number of elements of the set $S$.
\end{corollary}

\begin{proof}
The first equation holds by Theorem \ref{Row}. We now prove the second equation. Assume, by a contradiction, that $|\{a_n | n \geq 0\} \cap \{b_n | n \geq 0\}| \geq 3$. Let $x, y,z$ be three elements in the intersection $\{a_n | n \geq 0\} \cap \{b_n | n \geq 0\}$ such that $x < y < z$ . Then there exist $a_1 \leq x \leq b_1, a_2 \leq y \leq b_2,a_3 \leq z \leq b_3$ such that the six positions $(a_1, x)$, $(x,b_1)$, $(a_2,y)$, $(y,b_2)$, $(a_3, z)$, and $(z,b_3)$ all have Sprague-Grundy value $k$. We have then $a_1 = b_1$, $a_2 = b_2$, and $a_3 = b_3$. Note that $a_2, a_3 > 0$ and so one can move from $(a_3,b_3)$ to $(a_2, b_2)$. This is a contradiction as these two positions belong to the $k$-sequence.
\end{proof}

%====================================================================================================================
%====================================================================================================================
%====================================================================================================================
% QUESTIONS AND CONJECTURES

We now return to a discussion of the Sprague-Grundy values in each row of the array $\mathbb{A}$. A sequence $(s_i)$ is said to be \emph{ultimately additively periodic} if there exist $N,p > 0$ such that $s_{n+p} = s_n+p$ for all $n \geq N$. Based on our computer explorations, we conjecture that this periodicity holds for each row in the array $\mathbb{A}$.

%====================================================================================================================
\smallskip
\begin{conjecture} \label{add-per}
Let $a \geq 0$. The sequence $\{\G(a,n)_{n \geq 0}\}$ is ultimately additively periodic.
\end{conjecture}

Recall that ultimately additive periodicity also holds for Wythoff's game \cite{Dress, landman}. Our computer explorations show that this periodicity is common in variants of Wythoff's game.  (See \cite{ho}.) This lead us to the following problem.

%====================================================================================================================
\smallskip
\begin{problem} \label{cha-add-per}
Characterize variants of Wythoff's game whose nim-sequences $\{\G(a,n)\}_{n \geq 0}$ are ultimately additively periodic for all $a$.
\end{problem}

We now discuss the distribution of Sprague-Grundy values of $\F$-Wythoff in each diagonal parallel to the main diagonal in the array $\mathbb{A}$. It is well known that each such diagonal for Wythoff's game contains every nonnegative integer \cite{blass}. Based on our computer explorations, the same result is conjectured for $\F$-Wythoff.

%====================================================================================================================
\smallskip
\begin{conjecture} \label{Dia}
Let $a, g$ be nonnegative integers. There exists a unique integer $b$ such that $\G(b,a+b) = g$.
\end{conjecture}

Note that $\G(\lfloor \phi a \rfloor + 1,\lfloor \phi a \rfloor+a + 1) = 0$ and so Conjecture \ref{Dia} holds for $g = 0$. We give here the proof of Conjecture \ref{Dia} for the case $a = 0$.

\begin{proof}[Proof of Conjecture \ref{Dia} for the case $a = 0$]
Assume that $g > 0$ and assume by contradiction that the sequence $\{\G(n,n)\}_{n \geq 0}$ does not contain $g$. We can assume that $g$ is the smallest integer not in that sequence. Then there exists the smallest integer $b_0 > 0$ such that
\begin{align} \label{BW-b00}
\{0,1,\ldots, g-1\} \subseteq \{\G(i,i) | i \leq b_0-1\}.
\end{align}

For each $s \leq b_0$, there exists at most one value $t_s \geq b_0$ such that $\G(s,t_s) = g$. Let $S$ be the set of the values $t_s$, and set
\[
T_0 = \begin{cases}\max(S),&\ \text{if}\ S\not=\emptyset;\\
b_0, &\ \text{otherwise}.
\end{cases}
\]
Then $\G(s,t) \neq g$ for $s \leq b_0, t > T_0$. Note that $\G(i,i) \neq g$ for all $i$. Set $m = T_0+1$. We have
\[
\G(m,m) = mex\{\G(m,i), \G(j,j) | i, 1 \leq j \leq m-1\}.
\]
By (\ref{BW-b00}), $\G(m,m) \geq g$ and so $\G(m,m) > g$ as $\G(m,m) \neq g$. Since $\G(j,j) \neq g$ for all $j$, there exists $i_0 \leq m-1$ such that $\G(m,i_0) = g$. Note that $i_0 > b_0$ as otherwise $m \leq T_0$ giving a contradiction with $m = T_0+1$. We have
\[
\G(i_0,i_0) = mex\{\G(i_0,j), \G(l,l) | j \leq i_0-1, 1 \leq l \leq i_0-1\}.
\]
By (\ref{BW-b00}), $\G(i_0,i_0) \geq g$ and so $\G(i_0,i_0) > g$ as $\G(i_0,i_0) \neq \G(i_0,m) = g$ by Theorem \ref{Row}. Since $\G(l,l) \neq g$ for all $l$, there exists $j_0 < i_0$ such that $\G(i_0,j_0) = g$. However, there exists a move from $(m,i_0)$ to $(i_0,j_0)$ as $m > i_0 > j_0$. This is a contradiction.
\end{proof}

So far, we haven't been able to prove the conjecture for any given $a \geq 1$.

%====================================================================================================================
%====================================================================================================================
%====================================================================================================================
\section{$\F$-Wythoff in mis\`{e}re play}

Recall that in the game we have discussed so far, a player wins if (s)he makes the last move. This is the \emph{normal} convention. Oppositely, in \emph{mis\`{e}re} convention, a player is declared to be the winner if (s)he forces the opponent to make the last move.

In this section, we study $\F$-Wythoff played under the mis\`{e}re convention. We show that $\F$-Wythoff and mis\`{e}re $\F$-Wythoff swap Sprague-Grundy values 0 and 1 while agreeing for all other Sprague-Grundy values. We first show that the $\P$-positions of mis\`{e}re $\F$-Wythoff are exactly those positions which have Sprague-Grundy value 1 of $\F$-Wythoff while those positions which have Sprague-Grundy value 1 of mis\`{e}re $\F$-Wythoff are exactly $\P$-positions of $\F$-Wythoff. The proofs of the following two theorems are essential the same as those of Theorems \ref{FW-P} and \ref{V1}, respectively, and so we omit the proofs.

%====================================================================================================================
\smallskip
\begin{theorem} \label{P-Mis}
The position $(a,b)$ with $a \leq b$ is a $\P$-positions in mis\`{e}re $\F$-Wythoff if and only if it is an element of the set
\[\{(0,1), (\lfloor \phi n \rfloor +2, \lfloor \phi^2 n \rfloor +2) | n \geq 0\}.\]
\end{theorem}

%====================================================================================================================
\smallskip
\begin{theorem} \label{1-Mis}
The position $(a,b)$ with $a \leq b$ has Sprague-Grundy value 1 in mis\`{e}re $\F$-Wythoff if and only if it is an element of the set
\[\{(\lfloor \phi n \rfloor + 1, \lfloor \phi^2 n \rfloor + 1) | n \geq 0\}.\]
\end{theorem}

We now go further to show that $\F$-Wythoff and its mis\`{e}re version differ on those positions which have Sprague-Grundy value 0 and 1 via a swap. An impartial game can be described as a finite directed acyclic graph without multiple edges in which each vertex is a position and each downward edge is a move. Note that if a game $G$, under the normal convention, is described as a graph $\Gamma$, then the mis\`{e}re version of $G$ can be described as the graph $\Gamma^-$ obtained from $\Gamma$ by adding one extra vertex $v$ and an edge downward from each final vertex (vertex without outgoing edge) in $\Gamma$ to $v$.

For an impartial game $G$, denote by $\G_G$ and $\G^-_G$ the Sprague-Grundy functions for $G$  and its mis\`{e}re version, respectively. If there exist some subset $V_0$ of $\P$-positions and $V_1$ of those positions which have Sprague-Grundy value 1 of $G$ such that $\G_G(p) + \G^-_G(p) = 1$ if $p \in V_0 \cup V_1$ and $\G_G(p) = \G^-_G(p)$ otherwise, then $G$ is said to be \emph{miserable}. If $V_0$ coincides with the $\P$-positions and $V_1$ coincides with those positions which have Sprague-Grundy value 1 of $G$, then $G$ is said to be \emph{strongly miserable}. Several miserable and strongly miserable impartial games are studied in \cite{RRR}, including Wythoff's game. Gurvich has shown that Wythoff's game is miserable but not strongly miserable \cite{RRR}. We now show that strong miserability holds for $\F$-Wythoff.

%====================================================================================================================
\smallskip
\begin{theorem} \label{Str.Mis}
The game $\F$-Wythoff is strongly miserable.
\end{theorem}

\begin{proof}
Let $\Gamma$ be the graph of $\F$-Wythoff. Consider the graph $\Gamma^-$ of mis\`{e}re $\F$-Wythoff obtained from $\Gamma$ with the extra sink $v_0$. For each vertex (position) $v$, the {\em height} $h(v)$ of $v$ is the length of the longest directed path from $v$ to the sink $v_0$. Denote by $\G^-$ the Sprague-Grundy function for mis\`{e}re $\F$-Wythoff. We will prove by induction on $h(v)$ that $\G^-(v) = \G(v)$ if $\G^-(v) \geq 2$. One can check that the claim is true for $h(v) \leq 2$. Assume that the claim is true for $h(v) \leq n$ for some $n \geq 2$. We show that the claim is true for $h(v) = n+1$. For each $k < \G^-(v)$, there exists $w_k$ such that $\G^-(w_k) = k$ and one can move from $v$ to $w_k$. By Theorems \ref{P-Mis}, \ref{1-Mis} and the inductive hypothesis, we have
\[\{\G(w_k)| 0 \leq k < \G^-(v)\} = \{0,1, \ldots, \G^-(v)-1\}.\]
Note that if there exists a move from $v$ to some $w$ in $\F$-Wythoff, then that move can also be made in mis\`{e}re $\F$-Wythoff. Moreover, by Theorems \ref{P-Mis}, \ref{1-Mis} and the inductive hypothesis, $\G(w) \neq \G^-(v)$. We have
\[\G(v) = mex\{\G(w) | \text{$w$ is a follower of $v$}\}.\]
Since the $mex$ set includes the set $\{0,1, \ldots, \G^-(v)-1\}$ but excludes $\G^-(v)$, $\G(v) = \G^-(v)$.
\end{proof}

Recall that Wythoff's game and several of its variants are either miserable or strongly miserable \cite{RRR}. Our computer explorations show that the two variants of Wythoff's game recently discussed in \cite{ho} are miserable. This commonness leads us to the following question and problem.

%====================================================================================================================
\smallskip
\begin{question}
Are all extensions of Wythoff's game either miserable or strongly miserable?
\end{question}

%====================================================================================================================
\smallskip
\begin{problem}
Characterize miserable or strongly miserable restrictions of Wythoff's game.
\end{problem}

%====================================================================================================================
%====================================================================================================================
%====================================================================================================================
\section{Variants of $\F$-Wythoff preserving its $\P$-positions}

In this section, we answer the question as to whether there exists either a restriction or an extension of $\F$-Wythoff preserving its $\P$-positions. For an impartial game, a move is said to be \emph{redundant} \cite{Ext-Res} if the elimination of this move from the game does not change the set of $\P$-positions. A move is therefore not redundant if there exists a position $p$ such that that move is the unique winning move from that position. Given an impartial game, a move can be added into the set of moves without changing the set of $\P$-positions if this move does not lead a $\P$-position to another $\P$-position. We will introduce in this section one restriction and one extension of $\F$-Wythoff preserving it $\P$-positions. The idea for this section comes from our recent work on variants of Wythoff's game preserving its $\P$-positions \cite{ho}.

Consider the restriction of $\F$-Wythoff which we call \emph{$\F_{\R}$-Wythoff} obtained as follow: if the two piles have different sizes, removing tokens from a single pile cannot be made on the smaller pile. The second game is an extension of $\F$-Wythoff obtained by adding an extra move as follows: from a position $(a,b)$ with $a \leq b$, one can remove $k$ tokens from the pile of size $a$ and remove $l \leq k$ tokens from the pile of size $b$ provided that the integer ratio of the two entries does not change. We call this extension \emph{$\F_{\E}$-Wythoff}. The proofs for the results in this section are quite similar to those for $\F$-Wythoff and we leave them for the reader.

%=============================================================================================== =====================
\smallskip
\begin{theorem} \label{REF-P}
The $\P$-positions of $\F_{\R}$-Wythoff $($and $\F_{\E}$-Wythoff$)$ are identical to those of $\F$-Wythoff.
\end{theorem}

We now answer the question as to whether there exists a restriction of $\F_{\R}$-Wythoff preserving its $\P$-position.

%====================================================================================================================
\smallskip
\begin{theorem} \label{RF-Res}
There is no restriction of $\F_{\R}$-Wythoff preserving its $\P$-positions
\end{theorem}

\begin{proof}
We will show that neither of the moves in $\F$-Wythoff is redundant. We need to show that for every positive integer $k$, the following two properties hold:
\begin{itemize}
\item [(i)]  there exists a winning position $(a,b)$ with $a < b$ such that removing $k$ tokens from the larger pile is the unique winning move;
\item [(ii)] there exists a winning position such that removing $k$ tokens from both piles is the unique winning move.
\end{itemize}

For (i), let $a = 2, b = 3+k$. Then $(a,b)$ is an $\N$-position. Moreover, by Theorem \ref{FW-P}, removing $k$ tokens from the larger pile is the unique winning move.

For (ii), we first claim that there exist positive integers $n$ and $m$ such that $\lfloor\phi n\rfloor + k = \lfloor\phi m\rfloor$. In fact, set $n_1 = \lfloor 2\phi\rfloor = 3$, $n_2 = \lfloor3\phi \rfloor = 4$, $m_1 = 3+k$, and $m_2 = 4 + k$. We show that either $m_1$ or $m_2$ is of the form $\lfloor\phi m\rfloor$ for some $m$. Assume by contradiction that neither $m_1$ nor $m_2$ is of the form $\lfloor\phi m\rfloor$. By Lemma \ref{Comp}, there exist $r_1 < r_2$ such that $m_1 = \lfloor\phi r_1\rfloor +r_1$, $m_2 = \lfloor\phi r_2\rfloor +r_2$. Note that $\lfloor\phi r_1\rfloor < \lfloor\phi r_2\rfloor$ and so
\[1 = m_2 - m_1  = \lfloor\phi r_2\rfloor +r_2 - (\lfloor\phi r_1\rfloor +r_1)
                 = \lfloor\phi r_2\rfloor - \lfloor\phi r_1\rfloor + r_2 - r_1 \geq 2 \]
giving a contradiction. Now, if $m_1 = \lfloor\phi m\rfloor$ (resp. $m_2 = \lfloor\phi m\rfloor$), let $n = 2$ (resp. $n = 3$). Then $n$ and $m$ satisfy the condition $\lfloor\phi n\rfloor + k = \lfloor\phi m\rfloor$.

Let $a = \lfloor\phi n\rfloor +1 + k$, $b = \lfloor\phi n\rfloor +n +1 + k$. Then $(a,b)$ is an $\N$-position and removing $k$ tokens from both piles is a winning move. (Note that an equal number of tokens can be removed from $(a,b)$ by Lemma \ref{U-lem}.) Moreover, for $k' \neq k$, removing $k'$ tokens from both piles is not a winning move. In fact, otherwise there is a move between the two $\P$-positions $(a-k,b-k)$ and $(a-k',b-k')$. (Note that $\lfloor b-k/a-k \rfloor = \lfloor b-k'/a-k' \rfloor = 1$ by the last equation of Lemma \ref{U-lem}.) It remains to show that removing $k$ tokens from both pile of $(a,b)$ is the unique winning move. Assume by contradiction that there exists another winning move from $(a,b)$. This move must take some $l$ tokens from the larger pile leading $(a,b)$ to some position $(\lfloor\phi r\rfloor + 1, \lfloor\phi r\rfloor + r + 1)$. First consider the case $b-l = \lfloor\phi r\rfloor + 1$, $a = \lfloor\phi r\rfloor + r + 1$. We have shown the existence of $m$ such that $\lfloor\phi m\rfloor = \lfloor\phi n\rfloor + k = a - 1$ and so $\lfloor\phi m \rfloor = \lfloor\phi r\rfloor + r = \lfloor\phi^2 r\rfloor$ which contradicts Lemma \ref{Comp}. Now consider the case $b-l = \lfloor\phi r\rfloor + r + 1$, $a = \lfloor\phi r\rfloor + 1$. We have
\begin{align*}
\begin{cases}
a = \lfloor\phi n\rfloor + k + 1 = \lfloor\phi r\rfloor + 1,\\
b-l = \lfloor\phi n\rfloor + n + 1 + k - l = \lfloor\phi r\rfloor + r + 1.
\end{cases}
\end{align*}
The first equation implies $n < r$. By substituting $\lfloor\phi n\rfloor + k$ from the first equation into the second one, we get $\lfloor\phi r\rfloor + n - l = \lfloor\phi r\rfloor + r$ which implies $n = l+r > r$ giving a contradiction. Therefore, this case is impossible.
\end{proof}

%=============================================================================================== =====================
\smallskip
\begin{theorem} \label{REF-1}
The positions which have value 1 of $\F_{\R}$-Wythoff $($and $\F_{\E}$-Wythoff$)$ are identical to those of $\F$-Wythoff.
\end{theorem}

One can check that those positions which have Sprague-Grundy value $i$, for $2 \leq i \leq 3$, in $\F_{\R}$-Wythoff  and those positions which have Sprague-Grundy value $i$, for $2 \leq j \leq 7$, in $\F_{\E}$-Wythoff can be obtained from $\P$-positions of $\F$-Wythoff by a translation, except for some first positions.

Consider the 2-dimension arrays of Sprague-Grundy values of $\F_{\R}$-Wythoff and $\F_{\E}$-Wythoff (as in Table \ref{T}). We have similar results to Theorem \ref{Row}.
%=============================================================================================== =====================
\smallskip
\begin{theorem} \label{REF-1}
Let $a, g$ be nonnegative integer. For each of the two games $\F_{\R}$-Wythoff and $\F_{\E}$-Wythoff, there exists $b$ such that the position $(a,b)$ has Sprague-Grundy value $g$. The uniqueness holds for $\F_{\E}$-Wythoff.
\end{theorem}

We end this section with a result on the strong miserability of the two variants.

%====================================================================================================================
\smallskip
\begin{theorem} \label{Str.Mis-RE}
The two games $\F_{\R}$-Wythoff and $\F_{\R}$-Wythoff are both strongly miserable.
\end{theorem}

\section{More open questions}
Based on the translation phenomenon discussed above, we raise in this section two general questions on variants of Wythoff's game.

%====================================================================================================================
\smallskip
\begin{question}
Does there exist another variant of Wythoff's game whose $\P$-positions, possibly except for a finite number of positions, accept the formula $(\lfloor \phi n \rfloor+m, \lfloor \phi^2n \rfloor +m)$ for some $m \geq 2$.
\end{question}

More generally,
%====================================================================================================================
\smallskip
\begin{question}
Does there exist another variant of Wythoff's game whose $\P$-positions, possibly except for a finite number of positions, accept the formula $(\lfloor \phi n \rfloor+a, \lfloor \phi^2n \rfloor +b)$ for some integers $a$ and $b$.
\end{question}

%====================================================================================================================
%====================================================================================================================
%====================================================================================================================
%====================================================================================================================
\smallskip
\begin{ack}
I thank Graham Farr at Monash University for his suggestion of the game and for his communication which improved the paper.
\end{ack}

%====================================================================================================================
%====================================================================================================================
%====================================================================================================================

\end{document}